\documentclass{arXiv}

\title[Adaptive Online Non-stochastic Control]{Adaptive Online Non-stochastic Control}
\usepackage{times}
\coltauthor{\Name{Naram Mhaisen} \Email{n.mhaisen@tudelft.nl}\and
 \Name{George Iosifidis}  \Email{G.iosifidis@tudelft.nl}\\
 \addr Faculty of Electrical Engineering, Mathematics and Computer Science.\\Delft University of Technology. The Netherlands.}

\usepackage{nicefrac}
\usepackage{graphicx}
\usepackage{autonum}
\usepackage{amsmath,amsfonts,bm}
\newcommand{\sumT}{\sum_{t=1}^T}

\newcommand{\dtp}[2]{\langle {#1}, {#2} \rangle}
\newcommand{\grd}{\nabla}

\renewcommand{\vec}[1]{\bm{#1}}

\DeclareMathOperator*{\argmin}{argmin}

\DeclareMathOperator*{\minimize}{minimize}
\usepackage{algorithm}
\usepackage{algorithmic}
\usepackage{comment}

\begin{document}

\maketitle

\begin{abstract}%
We tackle the problem of Non-stochastic Control (NSC) with the aim of obtaining algorithms whose policy regret is proportional to the difficulty of the controlled environment. Namely, we tailor the Follow The Regularized Leader (FTRL) framework to dynamical systems by using regularizers that are proportional to the actual witnessed costs. The main challenge arises from using the proposed adaptive regularizers in the presence of a state, or equivalently, a memory, which couples the effect of the online decisions and requires new tools for bounding the regret. Via new analysis techniques for NSC and FTRL integration, we obtain novel disturbance action controllers (DAC) with sub-linear data adaptive policy regret bounds that shrink when the trajectory of costs has small gradients, while staying sub-linear even in the worst case.%

\end{abstract}

\begin{keywords}%
  Non-stochastic control; Follow the Regularized Leader; Online learning.%
\end{keywords}

\section{Introduction}\label{sec:intro}
This paper tackles the Online Non-stochastic Control problem: find a policy that endures minimum cost while controlling a dynamical system whose state changes via an unknown combination of learner's actions and external parameters. Optimal Non-stochastic control has significant applications ranging from the control of medical equipment \citep{app-ventilator} to energy management in data centers \citep{lee2021online}. This work advances the results on this fundamental problem by proposing optimal control algorithms based on adaptive online learning. 

\subsection{Background \& Motivation}
% DS + DAC policy + Regret + current results
We consider a typical NSC problem with a time-slotted dynamical system \citep{pmlr-v97-agarwal19c}. Namely, at each time step, the controller observes the system state $\vec{x}_t \in \mathbb{R}^{d_x}$ and decides an action %$\vec{u}_t \in \mathcal{U} \subset \mathbb{R}^{d_u}$
$\vec{u}_t \in \mathbb{R}^{d_u}$ which induces cost $c_t(\vec{x}_t, \vec{u}_t)$. Then, the system transitions to state $\vec{x}_{t+1}$. Note that the new state and cost function, at each step, are revealed to the controller \emph{after} it commits its action. Similar to \citep{pmlr-v97-agarwal19c}, we study Linear Time Invariant (LTI) systems where the transition is parametrized by matrices $A$, $B$ and a \emph{disturbance} vector $\vec w_t$:
\begin{align}
        \label{eq:LTIdynamics}
    \vec{x}_{t+1} = A \vec{x}_{t} + B \vec{u}_{t} + \vec{w}_{t}.
\end{align}
We allow $\vec{w}_t$ to be arbitrarily set by an \emph{adversary} that aims to manipulate the state transition, and we only restrict it to be universally upper-bounded, i.e., $\| \vec{w} \| \leq w$. Similarly, the adversary is allowed to select at each step \emph{any} Lipschitz continuous convex cost function $c_t:\mathbb{R}^{d_x}\times \mathbb{R}^{d_u}\mapsto \mathbb{R}$.

The controller's task is to deduce a (possibly non-stationary) policy that maps states to actions, $\pi: \vec{x}\mapsto \vec{u}$, from a policy class $\Pi$, leading to a trajectory of low costs $\{c_t(\vec x_t, \vec{u}_t)\}_{t=1}^T$. 
The employed performance metric in this setting is the \emph{policy regret}, cf. \citep{hazan20a-us}, which measures the accumulated extra cost endured by the learner's policy compared to a stationary cost-minimizing policy designed with access to all future cost functions and disturbances:
% Technically, the policy regret is defined as:  
\begin{align}
    \label{eq:regret-first-def}
    \mathcal{R}_T \doteq \sumT c_t\left(\vec x_t, \vec u_t\right) 
    - \min_{\pi\in\Pi} \sumT
    c_t\big(\vec x_t(\pi), \vec u_t(\pi)\big)
    \vspace{-2mm}
\end{align}
where $\vec x_t(\pi), \vec u_t(\pi)$ are the \emph{counterfactual} state-action sequences that would have been generated under the benchmark policy, whereas $(\vec x_t, \vec u_t)$ are the \emph{actual} state-action pair that resulted from following possibly different series of policies. A sublinear regret $\mathcal{R}_T\!=\!o(T)$ guarantees that the cost endured by the learner will converge to that of the optimal policy, i.e., $\nicefrac{\mathcal{R}_T}{T}\rightarrow 0$, and the recent Gradient Perturbation Controller (GPC), proposed in \citep{pmlr-v97-agarwal19c}, attains indeed $\mathcal{R}_T\!=\!\mathcal{O}(T^{\nicefrac{1}{2}})$. GPC's performance is established via a reduction to the Online Convex Optimization (OCO) with Memory framework \citep{NIPS2015_38913e1d}, which in turn established sublinear regret via a reduction to the standard OCO framework \citep{MAL-018}. 

This paper also aims to reuse results from OCO in NSC, but targets stronger guarantees than the typical $\mathcal{O}(T^{\nicefrac{1}{2}})$ regret. Namely, we aim to reduce the NSC problem to the \emph{adaptive} OCO framework, cf. \citep{mcmahan-survey17}, via the Follow-The-Regularized-Leader (FTRL) algorithm, which enables \emph{adaptive} regret bounds. Adaptive bounds depend on the observed losses $\{c_t(\cdot,\cdot)\}_{t=1}^T$, and are of the form $\mathcal{O}\big((\sumT  g_t)^{\nicefrac{1}{2}}\big)$, where $g_t \doteq \|G_t\|^2$ and $G_t \doteq  \grd c_t(\vec x_t, \vec u_t)$.\footnote{The gradient is w.r.t. the policy's parametrization. Its existence and boundness depend on $\Pi$ as will be detailed later.} Hence, $\mathcal{R}_T$ remains sub-linear in $T$ in the worst case (since $\|G_t\|\leq g, \forall t$ for Lipschitz functions), but is much tighter when small gradients are observed, i.e., when the cost functions are ``easy"\footnote{For a convex function $f(\cdot)$, costs with small gradient norms induce small regret: $f(\vec{x})\! - f(\vec{y}) \leq \dtp{\grd f (\vec{x})}{\vec{x} - \vec{y}}$.}. Methods with such guarantees have received significant attention as they adapt to the environment (observed costs), and hence are less conservative than their non-adaptive counterparts \citep{duchi11a}. 

Alas, such stronger adaptive regret bounds are yet to be seen in the NSC framework, and this is not a mere artifact: adaptivity in stateful systems raises fresh challenges that require new technical solutions. Specifically, due to the system's memory, the cost at each time step depend not only on the adversary but also on past actions of the learner. This over-time coupling, which is of unknown intensity in the adaptive case, perplexes the optimization of the action at each step, and indeed the vast majority of OCO-based control approaches employ the non-adaptive variants of the learning algorithms (see discussion in Sec. 2). In light of the above, we ask the question: \emph{can we design algorithms with policy regret bounds that adapt to easy costs, while still maintaining sub-linear regret in all cases?} We answer in the affirmative and present a learning toolbox with policy regret bounds that are adaptive w.r.t. the adversity of the environment (disturbances and cost functions). 

\subsection{Methodology and Contributions}
The analysis approach of NSC algorithms relies on approximating the cost of the non-stationary policy $\sumT c_t\left(\vec x_t, \vec u_t\right)$ by the counterfactual stationary one $\sumT c_t\left(\vec x_t(\pi_t), \vec u_t(\pi_t)\right)$ with a sub-linear approximation term of order $\mathcal{O}(T^{\nicefrac{1}{2}})$. This is possible due to the structure of LTI systems, where the effect of far-in-the-past actions decays exponentially, deeming the cost similar to that of ``OCO with bounded memory" framework \citep{cesa2013online}. Then, Online Gradient Descend (OGD) is used to obtain an $\mathcal{O}(T^{\nicefrac{1}{2}})$ bound for the stationary policy regret, leading to a total policy regret of the same order. In our case, however, we adopt the more general FTRL framework, cf. \citep{mcmahan-survey17}, which makes it 
no longer apparent that such a reduction is possible. 
% we can approximate the cost of a stationary policy by the non-stationary one.
% no longer apparent that we can approximate the cost of a stationary policy by the non-stationary one.

In detail, FTRL operates on the principle that at each $t$, the learner optimizes its decision\footnote{As will be detailed later, the action $\vec{u}_t$ is expressed as a linear combination of the parameters $M$.} $M_{t+1}$ by minimizing the aggregate cost until $t$, plus a strongly convex regularization term that penalizes the deviation from $M_t$. It is known that when these regularization terms are adaptive (i.e., the strong convexity is incremented at each $t$ proportionally to the witnessed cost of $M_t$ ), we get bounds proportional to the observed costs. Now, the issue is that unlike GPC where, at each $t$, the distance between two previous decision variables $M_s$ and $M_{s+1}$, $s\in[1,t-1]$ is $(i)$ fixed and $(ii)$ of order $\mathcal{O}(T^{-1/2})$, FTRL's adaptive regularizers make the distance between $M_s$ and $M_{s+1}$ $(i)$ dependant on the costs witnessed until $s$ (i.e, not fixed), and $(ii)$ independent of the horizon $T$ (or the time $t$). As these two properties were essential for proving GPC policy bounds, we are faced with a new challenge that is specific to FTRL and NSC integration. However, we show that the $c_t\left(\vec x_t, \vec u_t\right)$ cost can still approximate, up to a diminishing error, the counterfactual cost of $c_t\left(\vec x_t(\pi_t), \vec u_t(\pi_t)\right)$. This approximation effectively enables regularization based on the witnessed cost at $t$ and eventually leads to the sought-after adaptive regret bound of the form $\mathcal{O}\big((\sumT  g_{t})^{\nicefrac{1}{2}}\big)$.

At a more conceptual level, it is interesting to observe that GPC and similar controllers already incorporated some notion of ``adaptivity" in the sense that the performance guarantee is with respect to a \emph{benchmark} that depends on the observed costs. This is
clearly a more refined and less conservative approach than the classical $H_\infty$ control framework that benchmarks
itself against the worst-case scenario \citep{karapetyan2022regret}. In this work, we make the next step, and not only use a cost-dependent
benchmark like GPC, but also the performance of the controller w.r.t. that benchmark is itself shaped by the observed costs and perturbations. This is in contrast to GPC where the difference from the benchmark (i.e., the regret bound) does not depend on the encountered costs and perturbations. This
additional layer of adaptability expedites the learning rates whenever the problem allows it (i.e., smaller or less volatile gradients are encountered). While the proposed adaptive controller is designed to benefit from easy environments, it does suffer performance degradation when the environment is indeed adversarial, but only in terms of constant factors. Hence, it provides a fail-safe tool for NSC.

\subsection{Notation}
We denote scalars by small letters and use $h_{a:b}$ for $\sum_{s=a}^b h_s$. Vectors are denoted by bold small letters and matrices by capital letters. Both can be indexed with time via a subscript. We denote by $\{\vec{a}_t\}_{t=1}^T$ the set $\{\vec{a}_1, \dots, \vec{a}_T\}$, and $[T]$ denotes  $\{1,2,\ldots,T\}$. When $T$ is not relevant, we use $\{\vec{a}_t\}_t$. Vector/matrices elements are indexed via a superscript in parenthesis, e.g., $\vec{a}_t^{(i)}$. $M\! =\! [M^{[i]}| M^{[j]}]$ denotes the augmentation of $M^{[i]}$ and $M^{[j]}$.  $\|\cdot\|$ denotes the $\ell_2$ norm for vectors and the Frobenius norm for matrices.  $\|\cdot\|_*$ is the dual norm. 
$\|\cdot\|_{\text{op}}$ is the matrix spectral norm (the induced $\ell_2$ norm). 
% The dot product between two vectors is denoted $\dtp{\cdot} {\cdot }$.
% and we abuse notation to write $\dtp{A} {B} \doteq \sum_{ij}A^{(i,j)}B^{(i,j)}$. 

\section{Related Work}
%\textbf{Online Non-stochastic Control.} 
The NSC thread was initiated in \citep{pmlr-v97-agarwal19c} which designed the first policy with sub-linear regret for dynamical systems, aspiring to generalize the classical control problem to general convex cost functions and arbitrary disturbances. In essence, the system state is modeled as a limited memory and the problem is cast into the standard OCO framework, which allows recovering the OCO bounds scaled by the memory length. Follow-up works refined these results for strongly convex functions \citep{simchowitz2020making,Agarwal-log,foster20b-log2}; and systems where the actions are subject to fixed or adversarially-changing constraints \citep{Li_Das_Li_2021, liu2023nsc-const}. NSC was also extended to systems where matrices $(A,B)$ are unknown \citep{hazan20a-us}, systems with bandit feedback \citep{GHH}, and time-varying systems \citep{gradu23a-ltv}. As expected, the regret bounds deteriorate in these cases, e.g., becoming $\mathcal{O}(T^{\nicefrac{2}{3}})$ for unknown systems and $\mathcal O(T^{\nicefrac{3}{4}})$ for systems with bandit feedback. All these works provide bounds that scale with the number of time steps $T$, as opposed to the fully adaptive bounds presented here, which are proportional to the witnessed costs and perturbations. 

The NSC framework was also investigated using variations of regret metric that e.g., compare the learner's decisions to the optimal one for each step $t$, or for an interval $\mathcal{I}\in [T]$. For instance, \citep{zhao2022non} recovered the $\mathcal O\big((TP_T)^{\nicefrac{1}{2}}\big)$ dynamic regret OCO bounds, where $P_T$ measures the times the optimal solution changes. Nonetheless, methods with static regret guarantees of the sort discussed in this paper, are building blocks for dynamic regret algorithms via, e.g., using them as ``base controllers" within  meta-controllers \citep{gradu23a-ltv, simchowitz20a-part-state}, and hence are directly relevant for these benchmarks, too. Competitive ratio is another metric that aims to minimize the ratio in accumulated cost between the benchmark policy and that of the learner. These algorithms have different semantics from the model considered here. For example, the adversary needs to reveal information about the cost function before the learner commits to a decision \citep{shi2020online}, or the cost is assumed to be a fixed quadratic function \citep{goel2022competitive}. Furthermore, it was shown in \citep{goel2023best} that having a regret guarantee against the optimal DAC policy  automatically guarantees a competitive ratio up to an additive sub-linear term. 
A recently refreshed line of work that also considers adversarial costs and transitions in stateful systems is Adversarial MDPs \citep{jin2023no}. Unlike NSC, this line considers finite states and actions, whose sizes appear in the bounds. Thus, their algorithms are irreducible to our settings. 

%\textbf{Adaptive Online Learning and NSC}
Adaptivity has been an important concept in OCO, and we refer the readers to the survey in \citep{mcmahan-survey17}, or the remarks in \citep[Sec. 3.5]{orabona2021modern} for details. In adaptive learning the regret bounds scale proportionally to the witnessed costs. Hence, for \emph{easy} environments with small costs, the bounds are considerably tighter than the standard worst-case ones which assume maximum cost at each round. On the other hand, if the environment follows actually a worst-case scenario, i.e., the adversary induces costs with large gradients that fluctuate aggressively, the adaptive bounds remain sublinear but have worse constant factors. Indeed, there is a price to be paid for adaptivity. For policy regret, the only form of adaptive bounds appears in \citep[Thm. 2]{zhang2022adversarial}, which, however, contain additive constants\footnote{These works also target the  ``adaptive regret" metric.} (i.e., do not directly collapse to $0$ even when all costs are $0$). Additionally, the authors do not consider the DAC policy class but the actions are rather a result of  a ``betting" technique that is suited to the ``tracking" sub-task in control. In fact, any DAC policy can be combined with their meta-algorithm to obtain guarantees on the ``tracking", and hence our works are complementary (See discussion in \citep[Sec 4.2]{zhang2022adversarial}). 
\section{Preliminaries}
\label{sec:preliminaries}
We introduce here the technical setup of NSC tailored to the assumptions and goals of this paper.
\textbf{Policy class.} We consider the class $\Pi^{}$ of Disturbance Action Controllers (DAC) that was introduced in \citep{pmlr-v97-agarwal19c}. A policy $\pi\in\Pi^{}$, with memory length $p$, is parametrized by $p$ matrices $M\doteq\big[M^{[1]}|M^{[2]}|\dots |M^{[j]}|\dots|M^{[p]}\big]$, and a fixed stabilizing controller $K$. We also define the set $\mathcal{M}\doteq\{M:\|M\| \leq \kappa_M \}$, which uses the standard bounded variable assumption. The action at a step $t$ according to a policy $\pi_t\in \Pi$, is then calculated via:
\begin{align}
    \label{eq:dac-action}
    \vec{u}_{t} = K \vec{x}_t + \sum_{j=1}^p M^{[j]}_t \vec{w}_{t-j}.
\end{align}

\textbf{Strong stability}. This assumption is standard in OCO-based control as it enables non-asymptotic analysis \citep[Def. 3.1]{cohen2018online}. It ensures the existence of a stabilizing controller $K$, such that $\|(A+BK)^t\|_{\text{op}}<\kappa (1\!-\delta)^t, \delta\in(0,1]$, $\kappa>0$. Here, we assume that $\|A\|_{\text{op}}\leq 1-\delta$, which allows us to satisfy the stability assumption with $K$ being the zero matrix. This simplification facilitates the analysis, but the obtained results are still extensible for nonzero $K$ since $K$ is an external parameter to NSC algorithms, see, e.g., the discussion in \citep[Remark 4.1]{zhang2022adversarial}. We also assume $\|B\|\leq \kappa_B$, while the boundedness of $\|A\|$ follows from its spectral norm bound. 

\textbf{Cost functions}. We consider the family of general convex functions for the losses, and we denote with $G_t(M)$ the gradient matrix\footnote{Since our decision variable is in $\mathcal{M}$, the gradient ``vector" can be organized into a matrix.} of the cost $c_t(\vec x, \vec u)$ w.r.t. $M$. If the argument $M$ is not relevant (e.g., $c_t$ is linear) or is fixed to $M_t$, then we denote the gradient simply with $G_t$. We also use the standard $l$-Lipschitz assumption $| c_t(\vec{x},\vec{y}) - c_t(\vec{x}',\vec{y}') | \leq l \| (\vec{x},\vec{u}) - (\vec{x}',\vec{u}') \|, \forall t\in [T].$

\textbf{DAC rationale}. The  DAC class strikes a balance between efficiency and performance. Specifically, both the states and actions are convex in the optimization variables $M$. This can be directly seen from \eqref{eq:dac-action} for the actions, whereas the state can be shown to be linear  by unrolling the dynamics equation in \eqref{eq:LTIdynamics} (e.g., see \citep[Lem. 4.3]{pmlr-v97-agarwal19c}) which we adapt below to our notation: 
\begin{lemma}
\label{lemma-dac-state}
Assuming that $\vec x_{1} =0$, and parameters $M_t, \vec w_t $ are $0$ for $t \leq 0$, the state of the system reached at $t+1$ upon the execution of actions $\{\bm u_i\}_{i=1}^t$, derived from a DAC policy $\pi_t$, is:
\begin{align}
\vec x_{t+1} = \sum _{i=0}^t A^i \bigg(\!B\sum_{j=1}^p \big(M^{[j]}_{t-i}\ \vec w_{t-i-j}\big) + \vec w_{t-i}\!\bigg).
%\!=\! \sum _{i=0}^t \! \alpha_i(t), 
\label{eq:state-nonstat-m}
\end{align}    
\end{lemma}
% why these policies? they approximate
Clearly, $c_t(\Vec{x}_t, \Vec{u}_t)$ is convex in $M$ since $\Vec{x}_t$ and $\Vec{u}_t$ are linear in variables $M$. This facilitates the minimization of the cost function.
Note that the boundness on the norm of $\vec x$ and $\vec u$,\footnote{The bound on the state is due to the strong stability assumption and will be explicit in the analysis.} along with the Lipschitzness of $c(\cdot,\cdot)$ implies the existence of $g$ s.t. $\|G_t\| \!\leq\! g, \forall t$.

Besides allowing convex costs, DAC policies are expressive as they approximate the class of linear policies up to an arbitrarily small constant error $\zeta$ . This follows from the next lemma from \citep[Lem. 6.9]{hazan-nsc-book}.

\begin{lemma}\label{lemma:class-approx}
Let $\|A\|_{\text{op}} =1-\delta, \|\vec w\| \leq w$. Then, for any linear policy $\pi^{\mathbb{L}}$ with $\|K\| \leq \kappa^{\mathbb{L}}$, and an arbitrarily small constant $\zeta$, there is a DAC policy with $p=\lfloor\nicefrac{1}{\delta}\log\left(\nicefrac{\kappa^{\mathbb{L}}w}{\delta\zeta}\right)\rfloor$ that achieves a cost at most $\mathcal{O}(\zeta)$ far from the cost of the linear policy:
$|c_t\big(\vec x_t(\pi), \vec u_t(\pi)\big) - c_t\big(\vec x_t(\pi^{\mathbb{L}}), \vec u_t(\pi^{\mathbb{L}})\big)|  = \mathcal{O}(\zeta)$.
\end{lemma}

\textbf{Policy regret}. We proceed to define formally the regret of a policy. We start with the benchmark DAC policy $\pi_\star\in \Pi$, that is fully characterized by matrix $M_\star$ which can be calculated by solving:
\begin{align}
\textbf{P}_1: \minimize_{M\in\mathcal{M}} \ \sumT c_t\big(\vec x_t(\pi_\star), \vec u_t(\pi_\star)\big),\  
\text{subject to} \quad \vec{x}_{t+1}\! &=\! A \vec{x}_{t} \!+ B \vec{u}_{t} + \vec{w}_{t}, \, \ \ \ \  \forall t\in [T],  \label{bhs-prob-const-state} 
\\
    \quad  \vec{u}_{t} \!&=\! K \vec{x}_t \!+\! \sum_{j=1}^p\! M^{[j]} \vec{w}_{t-j}, \, \forall t\!\in [T].\!\!  \label{bhs-prob-const-dac}
\end{align}
Constraints \eqref{bhs-prob-const-state} enforce the LTI dynamics of the state transition, and \eqref{bhs-prob-const-dac} ensure the actions are taken from a DAC policy. Clearly, this hypothetical policy can only be calculated with access to future disturbances and costs, whereas the learner, at each step $t$, has access only to information  until $t\!-\!1$. 

% We discuss next the policy regret of such a learner. 

There are some important notation remarks in order here. Recall that $\pi^{}_t$ is the policy at step $t$, where the learner decides $\vec u_t$ using $M_t$. The state reached at $t$ depends on the \emph{sequence} of policies $\pi_1,\pi_2,\dots,\pi_{t-1}$, also referred to as non-stationary policy, and we denote it with $\vec x_{t}\left(\pi_{1,\dots,t-1}\right)$. In contrast, the \emph{counterfactual} state reached at step $t$ by following a stationary policy $\pi$ at all steps up to $t$ is denoted $\vec x_t(\pi)$.\footnote{We call $\vec x_t(\pi)$ counterfactual since it is not necessarily equal to the true state of the system $\vec x_{t}\left(\pi_{1,\dots,t-1}\right)$.} Similarly, the  action at $t$ would in general differ for a stationary policy $\pi$ versus a non-stationary policy, i.e., $\vec u_t(\pi_t) \! \neq \! \vec u_t(\pi_{1,\dots,t})$. However, when $K$ is the zero matrix, these vectors are equivalent. 
To reduce clutter when possible, we omit the argument $\pi_{1,\dots,t-1}$ from $\vec{x_t}(\pi_{1,\dots,t-1})$. However, when we use the hypothetical state resulting from a fixed policy $\pi$, we make this explicit by writing $\vec{x}_t(\pi)$. The same applies to $\vec{u}_t$. Now, we define the policy regret as the cumulative difference between the cost induced by $\pi_\star$, and the cost of the non-stationary policy as: 
\begin{align}
&\mathcal{R}_T(\pi_{1,\dots,T},\pi_\star)\doteq \sumT \Big(c_t\big(\vec x_t(\pi_{1,\dots,t-1}), \vec u_t(\pi_{1,\dots,t})\big) 
    - 
    c_t\big(\vec x_t(\pi_\star), \vec u_t(\pi_\star)\big)\Big).
\end{align}
Because of Lemma \ref{lemma:class-approx}, $\pi_{1,\ldots, T}$ has also a regret guarantee against the best policy in the linear class. That is, there is a constant $a>0$, that depends only on $l$, $\kappa^{\mathbb{L}}$, and $\delta$, such that $\mathcal{R}_T(\pi_{1,\dots,T}, \pi^{\mathbb L}_\star) = \mathcal{R}_T(\pi_{1,\dots,T}, \pi_\star) + a\zeta T$. Hence, the sublinear regret rate can be preserved against any linear policy by tuning $\zeta$, which can be achieved by increasing the DAC memory parameter $p$ (see Lemma \ref{lemma:class-approx}). 
Next, we design algorithms that minimize $\mathcal{R}_T(\pi_{1,\dots,T}, \pi_\star)$. 

\section{\texttt{AdaFTRL-C}}
We propose an algorithm (\texttt{AdaFTRL-C}) for optimizing the policy parameters $M_t$, $t\in [T]$. The algorithm uses the FTRL update formula:
\begin{align}
    M_{t+1} = \argmin_{M \in \mathcal{M}} \left\{\sum_{s=1}^{t} c_s (\vec x_s(\pi), \vec u_s(\pi)) + r_s(M)\right\}. \label{eq:ftrl-general-step}
\end{align}
For the incremental regularizers $r_t(\cdot)$, we define\footnote{Unlike the discussion in the introduction, here we have $g_t$ as the max of the norm and its square. The necessity of this originates from lower bounds on OCO with switching costs \citep{gofer2014higher} and will become clear in the analysis.} $g_t \doteq \max (\|G_t\|,\|G_t\|^2)$ and use:
\begin{align}
&r_t(M) = \frac{\sigma_t}{2}\|M - M_t\|^2,
\ \text{where }\sigma_1 = \sigma\sqrt{g_1},\
\sigma_t = \sigma\big(\sqrt{g_{1:t}}\!-\!\sqrt{ g_{1:t-1}}\big)  \label{eq:memory-data-adap-regs-params}.
\end{align}
Note that the sum of the counterfactual cost sequence in \eqref{eq:ftrl-general-step} is indeed a function of $M$ since both $\vec x$ and $\vec u$ are written in terms of $M$ as shown earlier. Defining the norm $\|\cdot\|_{t} \doteq \sqrt{\sigma_{1:t}}\|\cdot\|$ we get that the regularizer $r_t(\cdot)$ is $1$- strongly convex w.r.t. norm $\|\cdot\|_{t}$, whose dual is $\|\cdot\|_{t,*}= \frac{1}{\sqrt{\sigma_{1:t}}}\|\cdot\|$.

Due to the cost linearization principle (see e.g., \cite[Sec. 2.4]{MAL-018}), \eqref{eq:ftrl-general-step} can be solved by two efficient steps: closed form solution of an unconstrained quadratic program, followed by a Euclidean projection, essentially matching the computational complexity of GPC.

Algorithm \ref{alg:adap-ftrl-c} summarizes the proposed routine; \texttt{AdaFTRL-C} first executes an action (line $3$). Then, the cost function is revealed and $G_t$ is recorded (line $4$). The system then transitions to state $\Vec{x}_{t+1}$, effectively revealing the disturbance vector $\Vec{w}_t$ (lines $5$, $6$). The strong convexity parameter is calculated (line $7$) and the next action is committed through updating the policy parameters (line $8$). The policy regret of this \texttt{AdaFTRL-C} routine is characterized by the following theorem:

\begin{algorithm}[!ht]
	\caption{\small{Adaptive FTRL Controller} (\texttt{AdaFRTL-C})}
	\label{alg:adap-ftrl-c}
	\begin{footnotesize}
    \textbf{Input}: An intrinsically stable LTI system $(A,B)$, $ \quad $$z \doteq pw\kappa_B$, $    \quad\sigma=\sqrt{\frac{\delta^2+2lz}{2\delta^2 k_M^2}}$, \quad  $M_1\in \mathcal{M}$ .  \\
    \textbf{Output}: $\{\boldsymbol{u}_t\}_{t=1}^T$
    \begin{algorithmic}[1] %[1] enables line numbers
        % \STATE Initialize $M_1\in \mathcal{M}$ arbitrarily.
        \FOR{$t=1,2,\ldots,T$}
            \STATE Use action $\vec{u}_t=\sum_{j=1}^p M_t^{[j]}\ \vec{w}_{t-j}$
            \STATE Observe cost $c_t(\vec{x_t}, \vec{u_t})$ and record the gradient  $G_t(M_t)$
            \STATE Observe new state $\vec{x_{t+1}}$
            \STATE Record the disturbance $\vec{w_t} = \vec{x}_{t+1} - A\vec{x_t} -B \vec{u_t}$
            \STATE Update regularization parameters $\sigma_t$ via \eqref{eq:memory-data-adap-regs-params}.
            \STATE Calculate ${M}_{t+1}$ via \eqref{eq:ftrl-general-step}.
        \ENDFOR
    \end{algorithmic}
\end{footnotesize}
\end{algorithm}
\vspace{-0.35cm}
\begin{theorem}\label{thm:gt-adap-ftrl}
Let $(A,B)$ be an intrinsically stable system (i.e., $\|A\|_{op}\!\leq\! 1\!-\!\delta$) with $\|B\|\leq \kappa_B$ and $\|\vec w\|\leq w$. Let the cost $c_t(\cdot,\cdot)$ be $l$-Lipschitz, and assume that the DAC parametrization $M=[M^{[1]},\dots, M^{[p]}]$ is bounded: $\|M\| \leq \kappa_M$. Define $z \doteq pw\kappa_B$ and $g_t\doteq\max (\|G_t\|,\|G_t\|^2)$. Algorithm \texttt{AdaFTRL-C} produces policies $\pi_{1,\dots,T}$ such that for all $T$, the following bound holds:
\begin{align}
    \mathcal{R}_T(\pi_{1,\dots,T}, \pi_\star) \leq 
\frac{2}{\delta}\sqrt{2\kappa_M^2\left(\delta^2+2lz\right)}\ \sqrt{\sumT g_t}\ \ .   
\end{align}
\end{theorem}

\textbf{Discussion.} \texttt{AdaFTRL-C} achieves a policy regret  that scales according to the \emph{witnessed} costs' gradients. This can be much tighter than the bounds of non-adaptive controllers, like GPC, in easy environments. The improvement depends on how smaller the $\|G_t\|$ values are than $g$ (an upper bound on $\|G_t\|$). In the worst case ($\|G_t\|=g, \forall t$), our bound is worse by a constant factor compared to GPC\footnote{Using OGD update rule instead would result in better \emph{worst-case} bound by a factor of $\sqrt{2g}$.}, but maintains its sub-linearity. This is expected since \texttt{AdaFTRL-C} is conservative in regularization. Hence, when the losses are large, its performance is slightly degraded.

\textbf{Proof layout. } First, we bound the difference between the learner's accumulated cost, which is induced by $\pi_{1,\dots,t}$, and the counterfactual accumulated cost had the learner followed $\pi_t$ in \emph{all} steps $1, \dots, t$ (i.e., stationary policy $\pi_t$). Second, since the cost $c_t$ is convex in the parametrization $M_t$ of such stationary policy, we leverage the FTRL theory to bound the regret against the parameters $M_\star$.
\begin{lemma}\label{lemma:state-dev}
Let $c_t\left(\vec x_t, \vec u_t\right)$ be the cost induced by  following $\pi_{1,\dots,t}$, and $c_t\big(\vec x_t(\pi_{t}), \vec u_t(\pi_{t})\big)$ the cost induced by following the policy $\pi_t, \forall \tau\!\leq\! t$, where each $\pi_t$ is derived using \eqref{eq:ftrl-general-step}. Let $z \doteq pw\kappa_B$ Then:
    \begin{align}
        {\nu}_{t} \doteq \bigg|c_t\left(\vec x_t, \vec u_t\right) - c_t\big(\vec x_t(\pi_{t}), \vec u_t(\pi_{t})\big) \bigg| \leq lz\sum_{i=0}^{t-1} (1-\delta)^{i} \sum_{\tau=t-i-1}^{t-1}   \frac{g_\tau}{\sigma\sqrt {g_{1:\tau}}}  \doteq \widehat{\nu}_{t}.
    \end{align}
\end{lemma}
\begin{proof}
Since $c_t$ is $l$-Lipschitz, it suffices to bound the distance between its arguments; and because it holds\footnote{Recall the stability with $K=0$ assumption. Otherwise, action deviation is reducible to state deviation from \eqref{eq:dac-action}.} $\vec{u}_t(\pi_{1,\dots,t})\!=\!\vec{u}_t(\pi_{t})$, we only need to bound $\|\vec{x}_{t} \!-\! \vec{x}_t(\pi_{t}) \|$. From Lemma \ref{lemma-dac-state}:
   \begin{align}
    \vec{x}_{t} \!\!=\!\! \sum _{i=0}^{t-1} A^i \bigg(B\ \sum_{j=1}^p M^{[j]}_{t-i-1}\vec w_{t-i-j-1} + \vec w_{t-i-1}\bigg) 
 =    \sum _{i=0}^{t-1} A^i \big(BM_{t-i-1}\ \vec w_{t-i-1,p} + \vec w_{t-i-1}\big), \label{eq:state_under_nonstationary} 
\end{align}
where we defined $\vec w_{t-i,p}\doteq (\vec w_{t-i-1}, \dots, \vec w_{t-i-p})$ so as to express the vector $\sum_{j=1}^p M_t^{[j]} \ \vec w_{t-j}$ equivalently as $M_t\ \vec w_{t-1,p}$. Similarly, we can write:
\begin{align}
    \label{eq:state_under_stationary}
    \vec{x}_t(\pi_{t}) =  \sum _{i=0}^{t-1} A^i \big(BM_t\ \vec w_{t-i-1,p} + \vec w_{t-i-1}\big). 
\end{align}
Note that we have $M_t$ in \eqref{eq:state_under_stationary} instead of $M_{t-i-1}$ in \eqref{eq:state_under_nonstationary}. This is because for the counterfactual state $\vec x_t $, $M_{t-i-1} = M_t, \forall i<t$.
Now, subtracting the above two state expressions, we have: $\left \|\vec{x}_{t}- \vec{x}_t(\pi_{t}) \right\|$ 
\begin{align}
 &\stackrel{(a)}{\leq}
 \sum_{i=0}^{t-1} \|A^i B  ({M}_{t-i-1} - {M}_{t}) \|_{\text{op}} \|\vec{w}_{t-i-1,p}\| \stackrel{(b)}{\leq}
 \sum_{i=0}^{t-1} \|A\|^i_{\text{op}} \| B \|_{\text{op}}  \|({M}_{t-i-1} - {M}_{t})\|_{\text{op}} \|\vec{w}_{t-i-1,p}\|
 \\
&\stackrel{(c)}{\leq}
 \sum_{i=0}^{t-1} (1-\delta)^{i} \| B \|  \|({M}_{t-i-1} - {M}_{t})\| \|\vec{w}_{t-i-1,p}\|
 \stackrel{(d)}{\leq}pw\kappa_B\sum_{i=0}^{t-1} (1-\delta)^{i}  \| {M}_{t-i-1} - {M}_{t} \| 
 \\
 & \stackrel{(e)}{=}z\sum_{i=0}^{t-1} (1-\delta)^{i}  \| {M}_{t-i-1} - {M}_{t} \| \stackrel{(f)}{\leq} z\sum_{i=0}^{t-1} (1-\delta)^{i}  \sum_{\tau=t-i-1}^{t-1}   \frac{\|G_{\tau}\|}{\sigma_{1:\tau}}
 \stackrel{(g)}{\leq} z\sum_{i=0}^{t-1} (1-\delta)^{i}\!\!\!\sum_{\tau=t-i-1}^{t-1}\!\!\frac{g_\tau}{\sigma\sqrt {g_{1:\tau}}}
 \label{eq:v_t-state-dev1},
\end{align}
where $(a)$ follows by the triangular inequality and $\|A\vec x\| \leq \|A\|_{\text{op}} \|\vec x\|$, $(b)$ by the sub-multiplicative property of matrix norms,  $(c)$ by the bound on the spectral norm of $A$ and  $\|\cdot\|_{op}\leq \|\cdot\|$, $(d)$ by $\|B\| \leq \kappa_B$, $\|\vec w_{t,p}\|\leq \sum_{b=1}^p  \|\vec{w}_{t-b}\|=pw$,  $(e)$ by grouping  $z \doteq pw\kappa_B$, $(f)$ by the auxiliary Lemma \ref{lemma-McMahan-instead}, and finally $(g)$ by the definition of $g_t$. Using $l$-Lipschitzness completes the proof.
\end{proof}
Now that we have characterized the  stationary vs. non-stationary cost deviation, we proceed into $(i)$ bound the cost of the stationary policy $(ii)$ bound this deviation term.
\begin{proof} \textbf{of Theorem \ref{thm:gt-adap-ftrl}}. By the definition of $\nu_t$ (in Lemma \ref{lemma:state-dev}), we can directly bound the policy regret:
\begin{align}    
\sumT \Big( c_t(\vec x_t, \vec u_t) - c_t\big(\vec x_t(\pi_\star), \vec u_t(\pi_\star)\big)\Big) &\leq \sumT\Big( c_t\big(\vec x_t(\pi_{t}), \vec u_t(\pi_{t})\big) - c_t\big(\vec x_t(\pi_\star), \vec u_t(\pi_\star)\big)\Big) + {\nu}_{1:T}
\end{align}
For the first sum, note that $c_t\left(\vec x_t(\pi_t), \vec u_t(\pi_t)\right)$ is a function of only the iterates $M_t$ and recall that $M_t$ are computed according to \eqref{eq:ftrl-general-step}). Hence, we recognize this sum as simply the standard FTRL regret, which can be bounded using the known  bound  \citep[Thm. 2]{mcmahan-survey17}. Thus,
\begin{align}
    \mathcal{R}_T(\pi_{1,\dots,T}, \pi_\star)\leq \sumT r_t(M_\star) + \frac{1}{2}\sumT \big\|\grd c_t\big(\vec x_t(\pi_{t}), \vec u_t(\pi_{t})\big)\big\|_{t,*}^2 \ \ + \nu_{1:T}.
\end{align}
We proceed to bounding each of the three summations above.
% In fact, it is particular to using adaptive regularizations in statefull systems.
\begin{align}
    &(i):\sumT r_t(M_\star) = \sumT \frac{\sigma_t}{2}\| M_\star - M_t\|^2 \leq 2\kappa_M^2\sumT \sigma_t   \stackrel{(a)}{\leq} 2\sigma\kappa_M^2\sqrt{g_{1:T}},
\\
    &(ii):\frac{1}{2}\sumT \big\|\grd c_t\big(\vec x_t(\pi_{t}), \vec u_t(\pi_{t})\big)\big\|_{t,*}^2 \leq \frac{1}{2}  \sumT \frac{\|G_t\|^2}{\sigma_{1:t}}  \leq \frac{1}{2}\sumT \frac{g_t}{\sigma \sqrt{g_{1:t}}} \stackrel{(b)}{\leq} \frac{1}{\sigma}\sqrt{g_{1:T}}
\end{align}
where inequality $(a)$ is due to the telescoping sum of $\sigma_t$, and $(b)$ is due to the common tool \citep[Lemma 3.5]{AUER200248}, or its extended version \citep[Lem. 4.13]{orabona2021modern}). The $\nu_{1:T}$ term is especially challenging since the standard tools from OCO do not suffice. Namely, observe that we cannot use Auer's lemma again for $\widehat \nu_{1:T}$, since we are adding, at each $t$,  $t$ many terms, each divided by a different quantity. Hence, we prove a more general result in auxiliary Lemma \ref{lemma:new-ora-gen}, which, we believe, is of independent interest as it can be used to provide similar bounds in learning algorithms for systems with memory. With this new technical result at hand, we get:
\begin{align}
(iii): \nu_{1:T} \stackrel{(c)}{\leq} \widehat \nu_{1:T} = lz\sumT \sum_{i=0}^{t-1} (1-\delta)^{i}\!\!\!\! \sum_{\tau=t-i-1}^{t-1}\!\frac{g_\tau}{\sigma\sqrt {g_{1:\tau}}} \stackrel{(d)}{\leq} \frac{2lz}{\sigma\delta^2} \sqrt{g_{1:T}},
\end{align}
where $(c)$ is from the definition of $\widehat \nu_t$ and $(d)$ is due to using Lemma \ref{lemma:new-ora-gen} with $f(x)=\nicefrac{1}{\sqrt{x}}$, $a_t = g_t$, and $a_0=0$. Tuning $\sigma$ we get $\sigma\!\!=\!\!\sqrt{\frac{\delta^2+2lz}{2\delta^2 k^2}}$. Substituting in $(i), (ii), \text{and } (iii)$ we get the bound. 
\end{proof}
\vspace{-5mm}
\section{Numerical examples \& Conclusion}
We conclude with numerical examples of the implications of adaptivity in NSC. We consider an LTI system with $\vec{x}, \vec{u}\!\!\in\!\!\mathbb{R}^2$, $p\!\!=\!\!10$, and hence $M\!\!\in\!\! \mathbb{R}^{2\times20}$. The dynamics are $A=0.9\times\boldsymbol{I}_{2}, B= \boldsymbol{I}_2$, with perturbation of maximum magnitude of $w=\sqrt{2}$. We consider a linear cost $c_t=\dtp{\vec{\theta}_t}{\vec x_t}$, with $\vec{\theta}_t\in[-10,10]^2$ and hence $l=\sqrt{200}$. Note that for $G_t^{(i,j)}$, the coefficient of $M_t^{(i,j)}$ when $c_t(x_t(\pi),u_t(\pi))$ is written explicitly in terms of $M$, we have that $G_t^{(i,j)} \leq \theta^{(i)}_t \times \nicefrac{w_t^{j}}{0.1} =100$ and hence $g = 633$. $\vec \theta_t$ and $\vec w_t$ are set according to different scenarios in Fig. \ref{fig:ne}. In scenario $(a)$, the encountered environment ($\vec{\theta}_t$ and $\vec{w}_t$) are smaller than the worst-case costs (by a factor of $10$). \texttt{AdaFTRL-C} takes advantage of this and accelerates the optimization, leading to an improved average regret by roughly the same order. In fact, (a) represents a whole class of easy environments where the witnessed costs are smaller than their maximum values, and in all these cases, \texttt{AdaFTRL-C} outperforms GPC. In the worst-case scenario (b), the degradation of \texttt{AdaFTRL-C} reaches a maximum of only $3.9$ times that of GPC. Lastly, (c) serves to demonstrate that even in highly volatile environments and worst-case costs/perturbation, \texttt{AdaFTRL-C} closely matches the performance of GPC. Concluding, we see that adaptivity offers a lot of potential gains for an entire family of easy environments, with tolerable degradation in the (single) worst-case scenario. The implementation of \texttt{AdaFTRL-C} and the code to reproduce the experiments are available \href{https://github.com/Naram-m/NSC}{here}.
\begin{figure}[h]
	\centering
\includegraphics[width=0.91\textwidth]{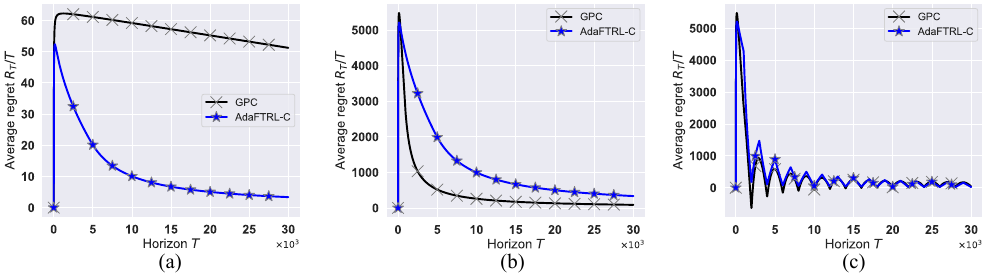}
\vspace{-2mm}
\caption{Average regret when $(a)$: $\vec\theta_t\!=\!(1,1), \vec w_t\!=\!-(0.1,0.1)\forall t$;\ $(b)$: $\vec\theta_t\!=\!(10,10), \vec w_t\!=\!-(1,1) \forall t$;  $(c)$: $\vec\theta_t\!=\!(10,10)$ or $-(10,10)$ (Alternating every $1000$ steps), $\vec w_t\!=\!-(1,1) \forall t$.}
	\label{fig:ne}
\end{figure}

\vspace{-10mm}
\section {Auxiliary Lemmas}
\begin{lemma}\label{lemma-McMahan-instead}
The distance between parameters $M_{t-i-1}$ and $M_t$, with $i\leq t-1$, where each $M_t$ is updated using \eqref{eq:ftrl-general-step}, with regularizers  \eqref{eq:memory-data-adap-regs-params}, can be bounded as $    \| {M}_{t-i-1} - {M}_{t} \| \leq \sum_{\tau=t-i-1}^{t-1}   \frac{\|G_{\tau}\|}{\sigma_{1:\tau}}$.
\end{lemma}
% \begin{align}
%     \| {M}_{t-i-1} - {M}_{t} \| \leq \sum_{\tau=t-i-1}^{t-1}   \frac{\|G_{\tau}\|}{\sigma_{1:\tau}}
% \end{align}

\begin{proof}
First, note that $\| {M}_{t-i-1} - {M}_{t} \| \leq \sum_{\tau=t-i-1}^{t-1}  \| {M}_{\tau} - {M}_{\tau+1} \|$. To bound each $\| {M}_{\tau} - {M}_{\tau+1} \|$, we leverage \citep[Lem. 7]{mcmahan-survey17} which states that given a convex function $\phi_1(\cdot)$ with a minimizer $ q_{1} \doteq \argmin_{ q}{\phi_1( q)}$ and a function $\phi_2(\cdot) \!= \phi_1(\cdot) \!+ \psi (\cdot)$ that is a strongly convex w.r.t a norm $\|\cdot\|$ and has minimizer $q_{2} \doteq \argmin_{ q}{\phi_2( q)}$, then we can bound the two minimizers as $ \|q_1\!-\!q_2\|\!\leq\| b\|_*$ for some (sub)gradient $b\in \partial \psi({q}_1)$. Now, we invoke this result by setting:
\begin{align}
&q_1=M_{\tau}, \ \ q_2={M}_{\tau+1}, \quad \phi_1(M) = \sum_{s=1}^{\tau-1}\Big( c_s \big(\vec x_s(\pi), \vec u_s(\pi_{})\big) + r_s(M)\Big) + r_{\tau}(M), \quad \text{and} \\
&\phi_2(M)= \sum_{s=1}^{\tau-1}\Big(c_s\big(\vec x_s(\pi), \vec u_s(\pi_{})\big) + r_s(M)\Big) + r_{\tau}(M)+\underbrace{ c_{\tau }\big(\vec x_{\tau}(\pi_{}), \vec u_{\tau}(\pi_{})\big)}_{\psi(M)}. 
\end{align}
The function $\phi_2(M)$ is strongly-convex w.r.t. the norm $\|\cdot\|_{\tau}=\sqrt{\sigma_{1:\tau}}\|\cdot\|$, a property inherited due to containing the sum of all regularizers up to step $\tau$, i.e., $\sum_{s=1}^{\tau}r_{s}(M)$. And the dual norm of the gradient of the above-defined $\psi(M)$ function, at each step $s$, is upper-bounded by $\|G_s\|_{\tau,*}$. 

Finally, it suffices to observe that: $(i)$ based on the definition of the update rule for variables $M$, and $(ii)$ the fact that $r_t(M), \forall t$ is a proximal regularizer thus $r_t(M_t)=0, \forall t$, we indeed have that $M_\tau = \argmin_{M\in\mathcal M} \phi_1(M)$ and $M_{\tau+1} = \argmin_{M\in\mathcal M} \phi_1(M) + \psi(M)$.
Now, applying \citep[Lem. 7]{mcmahan-survey17} we obtain $    \| {M}_{\tau} - {M}_{\tau+1} \|_\tau \leq \|G_\tau\|_{\tau,*} \rightarrow \| {M}_{\tau} - {M}_{\tau+1} \| \leq \frac{\|G_\tau\|}{\sigma_{1:\tau}}$.
\end{proof}

\begin{lemma} 
    \label{lemma:new-ora-gen}
    Let $a_\tau\geq 0\ \forall \tau$, $\delta \in (0,1]$ and $f:[0, \infty] \mapsto [0, \infty)$ be a non-increasing function. Then:
    \begin{align}
    \mathcal{S} \doteq \sumT \sum_{i=0}^{t-1} (1-\delta)^i \sum_{\tau = t-i}^t
    a_{\tau} \ f(a_{0:\tau}) \leq \frac{1}{\delta^2}\int_{a_0}^{a_{1:T}}{f(x)dx}
    \end{align}
\end{lemma}
\begin{proof}
Instead of summing over all $i < t$ for each $t \in [1,T]$, the sum  $\mathcal{S}$ can be equivalently re-written by summing over all $t > i$ for each $i \in [0,T]$:
\begin{align}
    \mathcal{S} = \sum_{i=0}^T (1-\delta)^i\sum_{t=i+1}^T \sum_{\tau = t-i}^t
    a_{\tau} \ f(a_{0:\tau}) \label{eq:i-sum}
\end{align}
Now we bound $ \mathcal{S}'\doteq\sum_{t=i+1}^T \sum_{\tau = t-i}^t a_{\tau} \ f(a_{0:\tau}) $. Let $s_t\doteq a_{0:t}$ 
\begin{align}
    &\sum_{\tau = t-i}^t
    \!\!a_{\tau}\!\! \ f(a_{0:\tau})   \stackrel{(a)}{=} \sum_{\tau=0}^i a_{t-\tau} f(s_{t-\tau}) \stackrel{(b)}{=} \sum_{\tau=0}^{i} \int_{s_{t-\tau-1}}^{s_{t-\tau}} f(s_{t-\tau})dx \stackrel{(c)}{\leq} \sum_{\tau=0}^{i} \int_{s_{t-\tau-1}}^{s_{t-\tau}} f(x)dx.
    \\
    &\text{Hence, }
    \mathcal{S}' \leq \sum_{\tau=0}^{i} \sum_{t=i+1}^T \int_{s_{t-\tau-1}}^{s_{t-\tau}}\!\!\!\! f(x)dx \stackrel{(d)}{=} \sum_{\tau=0}^{i}  \int_{s_{i-\tau}}^{s_{T-\tau}} \!\!\!\!f(x)dx 
    \stackrel{(e)}{=} \sum_{\tau=0}^{i}  \int_{s_{0}}^{s_{T}}\!\!\!\!f(x)dx
    =(i\!+\!1)\!\!\int_{s_{0}}^{s_{T}}\!\!\!\! f(x)dx.
\end{align}
Where $(a)$ follows by changing the sum index, $(b)$ is from $\int_{n}^m c\ dx = (m-n)c $ $,\forall n,m,c \in \mathbb{R}$, $(c)$ from $f(\cdot)$ being non-increasing, $(d)$ from the additive property of integrals, and $(e)$ from expanding the integration limits and positivity of $f(\cdot)$. Substituting the bound for $\mathcal{S}'$ back in \eqref{eq:i-sum} we get: 
\begin{align}
     \mathcal{S} \leq &\sum_{i=0}^T (1-\delta)^i (i+1)\!\!\int_{s_{0}}^{s_{T}}\!\!\!\!f(x)dx \leq \frac{1}{\delta^2}\int_{s_{0}}^{s_{T}} \!\!\!\!f(x)dx.
\end{align}
Where we used that $\sum_{i=0}^\infty\! i(1\!-\!\delta)^i\!\leq\!\nicefrac{1\!-\!\delta}{\delta^2}$ and $\sum_{i=0}^\infty (1\!-\!\delta)^i\!\!\leq\!\nicefrac{1}{\delta}$.
\end{proof}

\clearpage
\bibliography{nscref}
\clearpage
\section{Appendix}
\subsection{Proof of Lemma \ref{lemma-dac-state}}
Define 
\begin{align}
    \alpha_x(y) &\doteq \!  A^{x} \bigg( B\!\sum_{j=1}^p\big( M^{[j]}_{y-x}\ \vec{w}_{y-x-j}\!\big)+\vec w_{y-x}\bigg), x,y\in\mathbb{N}. 
\end{align}
so as to write the state as
\begin{align}
\vec x_{t+1}= \sum _{i=0}^t A^i \bigg(\!B\sum_{j=1}^p \big(M^{[j]}_{t-i}\ \vec w_{t-i-j}\big) + \vec w_{t-i}\!\bigg)\!=\! \sum _{i=0}^t \! \alpha_i(t)
\end{align}  
\begin{proof}
We prove the expression by induction. For $t=1$, \eqref{eq:state-nonstat-m} reduces to $\vec x_2=\vec w_1$ which follows directly from the dynamic equation in \eqref{eq:LTIdynamics} after substituting the assumptions on the initial state and actions. Then, assuming that \eqref{eq:state-nonstat-m} is true for  any $t$, we have that
\begin{align}
    \vec{x}_{t+2} &= A \vec{x}_{t+1} + B \vec{u}_{t+1} + \vec{w}_{t+1} 
    \\
    &= A \vec{x}_{t+1} + B \left(\sum_{j=1}^{p}M^{[j]}_{t+1}\ \vec w_{t+1-j}\right) + \vec{w}_{t+1} 
    \\
    &= A \sum _{i=0}^t A^i \left(B\sum_{j=1}^p M^{[j]}_{t-i} \vec w_{t-i-j} + \vec w_{t-i}\right)  + B \left(\sum_{j=1}^{p}M^{[j]}_{t+1}\vec w_{t+1-j}\right) + \vec{w}_{t+1} 
    \\
    &= A \sum _{i=1}^{t+1} A^{i-1} \left(B\sum_{j=1}^p M^{[j]}_{t-i+1} \vec w_{t-i-j+1} + \vec w_{t-i+1}\right)  + B \left(\sum_{j=1}^{p}M^{[j]}_{t+1}\vec w_{t+1-j}\right) + \vec{w}_{t+1} 
    \\
    &=  \sum _{i=1}^{t+1} \underbrace{A^{i} \left(B\sum_{j=1}^p M^{[j]}_{t-i+1} \vec w_{t-i-j+1} + \vec w_{t-i+1}\right)}_{\alpha_i (t+1)}  + \underbrace{B \left(\sum_{j=1}^{p}M^{[j]}_{t+1}\vec w_{t+1-j}\right) + \vec{w}_{t+1}}_{\alpha_{0}(t+1)} = \sum _{i=0}^{t+1} \alpha_i(t+1)
\end{align}
\end{proof}

\subsection{Proof of Lemma \ref{lemma:class-approx}}
To prove Lemma \ref{lemma:class-approx}, we need to make use of two known results in non-stochastic control. The first one, stated in Lemma \ref{lemma:state-under-linear} 
characterizes the state under any linear policy. The second, stated in Lemma \ref{lemma:state-deviation-two-policies}, relates the state deviation between DAC and linear policies to the deviation of their actions. 

\begin{lemma}
\label{lemma:state-under-linear}
Assuming that $\vec x_{1} =0$, and parameters $\vec w_t $ are $0$ for $t \leq 0$, the state of the system at $t+1$ upon the execution of actions $\{\bm u_s\}_{s=1}^t$ which are derived from the a linear policy parametrized by $K$: $\vec{u}_s = K \vec x_s $ can be written as: 
\begin{align}
    \vec x_{t+1} = \sum _{i=0}^t (A+BK)^i \vec w_{t-i}  = \sum _{i=0}^t \beta_i(t) \label{eq:state-nonstat-lin}, \quad
    \\
    \text{where for any } x,y\in\mathbb{N}, \quad \beta_x(y) \doteq  (A+BK)^{x}\ \vec w_{y-x} 
\end{align}
\end{lemma}
\begin{proof}
We prove the expression by induction. For $t=1$, \eqref{eq:state-nonstat-lin} reduces to $\vec x_2=\vec w_1$ which follows directly from the dynamic equation in \eqref{eq:LTIdynamics} after substituting the assumptions on the initial state and actions. Then, assuming that \eqref{eq:state-nonstat-lin} is true for some $t$, we have that
\begin{align}
    \vec{x}_{t+2} &= A  \vec{x}_{t+1} + B \vec{u}_{t+1} + \vec{w}_{t+1} = A \left(\sum _{i=0}^t (A+BK)^i \vec w_{t-i}\right)  + B \left(K \ \vec \sum _{i=0}^t (A+BK)^i \vec w_{t-i}\right) + \vec{w}_{t+1} 
    \\
    &= (A+BK) \sum _{i=0}^t (A+BK)^i \vec w_{t-i} + \vec{w}_{t+1} =  \sum _{i=1}^{t+1} \underbrace{(A+BK)^i \vec w_{t-i+1}}_{\beta_i(t+1)} + \underbrace{\vec{w}_{t+1}}_{\beta_0(t+1)} =    \sum _{i=0}^{t+1} \beta_i(t+1)
\end{align}
\end{proof}
Next, we need a lemma that characterizes the state deviation between any two policies 
\begin{lemma}\label{lemma:state-deviation-two-policies}
    let $\vec x_t({\pi_1})$ be the state of the system reached by following policy $\pi_1$ from the beginning of time. Analogously, let $\vec x_t({\pi_2})$ be the state resulting from following $\pi_2$. Let the system $(A,B)$ be intrinsically stable (i.e., $\|A\|_{\text{op}} \leq  (1-\delta)$, $\|B\| \leq 1$), and assume that the starting state and action are zero  $\vec{x}_1 = \vec 0, \vec{u}_1 = \vec 0$. Then, the following holds:
    \begin{align}
        \| \vec x_{t+1}({\pi_1}) - \vec x_{t+1}({\pi_2}) \| \leq  \sum_{i=0}^t \|A^{i}\|_{\text{op}}\|B \|\max_{j:j\leq t}  \|\vec{u}_{t-j}({\pi_1}) - \vec{u}_{t-j}({\pi_2})\| \leq \frac{1}{\delta} \max_{j:j\leq t}\|\vec{u}_{t-j}(\pi_1) - \vec{u}_{t-j}(\pi_2)\|. 
    \end{align}
In words, the deviation is fully controlled by the system stability and the maximum deviations of actions.
\end{lemma}
\begin{proof}
    We proceed by induction on $t$ to write the state $\vec{x}_{t+1}$ in terms of the previous disturbance and actions.
    The claim is that 
    \begin{align}
        \vec x_{t+1} = \sum_{i=0}^t A^i(B \vec u_ {t-i}+\vec{w}_{t-i})
    \end{align}
    For the base case of $t=1$, the above gives $\vec x_2 = \vec w_1$, which follows by the assumption on the initial action. Now, assuming that the statement is true for $t$, we have that for $t+1$: 
\begin{align}
    \vec{x}_{t+2} &= A \vec{x}_{t+1} + B \vec{u}_{t+1} + \vec{w}_{t+1} = A \left(\sum _{i=0}^t A^i(B \vec{u}_{t-i}+\vec{w}_{t-i}) \right)+ B \vec{u}_{t+1} + \vec{w}_{t+1}
    \\
    &\sum _{i=1}^{t+1} A^i(B \vec{u}_{t-i+1}+\vec{w}_{t-i+1}) + B \vec{u}_{t+1} + \vec{w}_{t+1} = \sum _{i=0}^{t+1} A^i(B \vec{u}_{t-i+1}+\vec{w}_{t-i+1}) 
\end{align}
Subtracting the state expression reached under \{$\vec{u}_s(\pi_1)\}_{s=1}^t$, and \{$\vec{u}_s(\pi_2)\}_{s=1}^t$, the result follows by the bound on the sum of geometric series.
\end{proof}

Now, we show that the actions, and consequently the states, produced by any DAC policy can approximate those of a linear policy with arbitrarily small error $\zeta$.
\begin{proof}[Proof of Lemma \ref{lemma:class-approx}]
let $\vec{u}_{t}(\pi)$ be the action produced by a stationary DAC policy $\pi$, and $\vec{u}_{t}(\pi^{\mathbb{L}})$ be the action produced by a linear policy $\pi^{\mathbb{L}}$. Then, by Lemma \ref{lemma:state-under-linear} we have
\begin{align}
 \vec{u}_{t}(\pi^{\mathbb{L}}) = \sum _{j=0}^{t-1} K (A+BK)^j \vec w_{t-j-1} &= \sum _{j=0}^{p-1} K (A+BK)^j \vec w_{t-j-1} + \sum _{j=p}^{t-1} K (A+BK)^j \vec w_{t-j-1} 
\\
&=\underbrace{\sum _{j=0}^{p-1} M^{[j+1]} \vec w_{t-j-1}}_{\vec u_{t}(\pi)} +  \sum _{j=p}^{t-1} K (A+BK)^j \vec w_{t-j-1}, 
\end{align}
 Where we denoted the $j$-th polynomial in $K$ with $M^{[j+1]}$ in the first sum. Now, note that the first sum is the DAC action. 
 Hence,
\begin{align}
     \|\vec{u}_{t}(\pi) - \vec{u}'_{t}(\pi^{\mathbb{L}})\| \leq \| \sum _{j=p}^t K (A+BK)^j \vec w_{t-j} \| &\stackrel{(\alpha)}{\leq} \sum _{j=p+1}^t (1-\delta)^j \kappa^{\mathbb{L}} w 
     \\
     &\stackrel{(\beta)}{\leq} \kappa^{\mathbb{L}} w \int_{j=p}^\infty e^{-\delta j} dj =  \kappa^{\mathbb{L}} \frac{w}{\delta} e^{-\delta p} \stackrel{(\gamma)}{\leq} \zeta
\end{align}
$(\alpha)$ is because $K$ is assumed a stabilizing linear controller $\|A+BK\|_{\text{op}}\leq 1-\delta$, 
$(\beta)$ is from $1+x\leq e^x$, and lastly $(\gamma)$ by the choice of $ p= \nicefrac{1}{\delta}\log\left(\nicefrac{\kappa^{\mathbb{L}}w}{\delta\zeta}\right)$. This small discrepancy in the actions translates to the same one in the states (up to the stability constant) by  Lemma \ref{lemma:state-deviation-two-policies}: 
\begin{align}
    \| \vec x_{t+1}({\pi}) - \vec x_{t+1}({\pi^{\mathbb{L}}}) \| \leq \frac{1}{\delta}\zeta 
\end{align}
Using the fact that $c_t(\cdot, \cdot)$ is Lipschitz $c_t\big(\vec{x}_t(\pi^{}), \vec{u}_t(\pi^{})\big) - c_t\big(\vec{x}_t(\pi^{\mathbb{L}}), \vec{u}_t(\pi^{\mathbb{L}})\big)=\mathcal{O}(\zeta)$
\end{proof}

\subsection{The non-adaptive case (OGD with fixed learning rate)}
When using OGD with fixed learning rate update as in \cite{pmlr-v97-agarwal19c}, we have that
\begin{align}
    \label{eq:ogd_step}
    M_{t+1} = M_t + \eta G_t
\end{align}
The proof of Lemma \ref{lemma:state-dev} follows the same steps until 
\begin{align}
    \nu_t \leq lz\sum_{i=0}^{t-1}(1-\delta)^i  \| {M}_{t-i-1} - {M}_{t} \| \leq lz\sum_{i=0}^{t-1}(1-\delta)^i \sum_{\tau=t-i-1}^{t-1} \|M_{\tau+1} - M_\tau\|
\end{align}
which we now bound from \eqref{eq:ogd_step} as 
\begin{align}
    lz\sum_{i=0}^{t-1} (1-\delta)^i \sum_{\tau=t-i-1}^{t-1} \|M_{\tau+1} - M_\tau\| \leq lz\eta\sum_{i=0}^{t-1} (1-\delta)^i (i+1)  \|G_\tau\| \leq  l z \eta g \sum_{i=0}^{t-1} (1-\delta)^i (i+1)  \leq \frac{l z \eta g}{\delta^2}
\end{align}
Hence we get that $v_{1:T} \leq \frac{l z \eta g}{\delta^2} T $.
To prove an analogous to Theorem \ref{thm:gt-adap-ftrl}, we have 
\begin{align}    
\sumT \Big( c_t(\vec x_t, \vec u_t) - c_t\big(\vec x_t(\pi_\star), \vec u_t(\pi_\star)\big)\Big) &\leq \underbrace{\sumT\Big( c_t\big(\vec x_t(\pi_{t}), \vec u_t(\pi_{t})\big) - c_t\big(\vec x_t(\pi_\star), \vec u_t(\pi_\star)\big)\Big)}_{\text{Regret of OGD-based stationary policy}} + {\nu}_{1:T}
\\
&\stackrel{(a)}{\leq} \frac{(2\kappa_M)^2}{2\eta} + \frac{\eta}{2}g^2T + v_{1:T} \leq \frac{2\kappa_M^2}{\eta} + \frac{\eta}{2}g^2T + \frac{l z \eta g}{\delta^2} T \label{eq:ogd_set_eta}
\end{align}
where inequality $(a)$ follows from the bound on OGD's regret \cite[Thm. 2.13]{orabona2021modern} (Recall that this sum is the regret of in the standard memoryless case), with $\eta_t \doteq \eta, \forall t$.
optimizing $\eta$ in \eqref{eq:ogd_set_eta}, we get $\eta = \frac{2 \kappa_M \delta}{\sqrt{g(g\delta^2+2lz)}}\frac{1}{\sqrt{T}}$, substituting back we get the non-adaptive OGD bound: 
\begin{align}
    \sumT \Big( c_t(\vec x_t, \vec u_t) - c_t\big(\vec x_t(\pi_\star), \vec u_t(\pi_\star)\big)\Big) \leq \frac{2\kappa_M}{\delta} \sqrt{g(g\delta^2+2lz)} \sqrt{T}
\end{align}

\subsection{On the choice of the decision set $\mathcal{M}$}
We note the different choice of the decision set $\mathcal{M}$ for DAC parametrization in \citep{pmlr-v97-agarwal19c} and some of the follow up papers, where $\mathcal{M}$ is defined as:
\begin{align}
    \mathcal{M} \doteq \left\{M = \Big[M^{[1]} | \dots |  M^{[j]}| \dots| M^{[p]}\Big] : \left\|M^{[j]}\right\| \leq \kappa_M(1-\delta)^j, \forall j\leq p\right\}.
\end{align}
I.e., the norm of the submatrices decays with parameter $j$. Such definition is necessary when analyzing the regret against the \emph{optimal linear controller}, where we have seen from Lemma \ref{lemma:class-approx} that $p\propto \log T$ is necessary, hence $\|\mathcal{M}\|$ \emph{increases with time}. Nonetheless, the exponential decay of the norm w.r.t $j$ still ensures bounded diameter in terms of $\kappa_M$. This  can be seen from e.g., \citep[Lem. 20, claim (iii)]{zhao2022non}. In our case, we analyze the regret against the \emph{optimal DAC} policy directly, and hence we use the definition: 
\begin{align}
    \mathcal{M} \doteq \left\{M = \Big[M^{[1]} | \dots |  M^{[j]} | \dots | M^{[p]}\Big] \ : \ \sum_{j=1}^p\left\|M^{[j]}\right\| \leq \kappa_M\right\},
\end{align}
which appears also in the monograph \citep[Sec. 6.2.4]{hazan-nsc-book}. Here,
$p$ is \emph{pre-determined} and fixed property of the class. Thus, we can use the diameter bound $\|M_1-M_2\| \leq 2\kappa_M$ as the set $\mathcal{M}$ is fixed. We leave to future work extending the analysis provided here to analyze the regret directly against the linear class, which is feasible given the tools developed in the paper.

\subsection{On the strongly stable controller $K$}
In this subsection, we discuss the implication of dropping the assumption $\|A\|_{\text{op}}\leq (1-\delta)$, which allowed us to have $K=0$ as a stabilizing  controller. 
We used the fact that $K=0$ is a stabilizing controller at the following points
\begin{itemize}
    \item In the cost deviation lemma (Lemma. \ref{lemma:state-dev}):  We used that $\vec u(\pi_{1,\dots,t}) = \vec u(\pi_{t})$ if $K=0$. In general $\vec u (\pi_{1,\dots,t})$ and $\vec u (\pi_{t})$ would differ only by $\|K\| \|\Vec{x_t} - \Vec{x_t}(\pi_t)\|$, this is exactly the state deviation term that we bound in the above mentioned lemma, and since $c_t(\vec u,\vec x)$ is Lipschitz in both arguments, the cost deviation can still be bounded but with different constant terms. 
    
    \item In Lemma \ref{lemma-dac-state}, where we write out the state $\vec{x_t}$ in terms of $M$ and $\Vec{w}$. Having a non-zero $K$ would result in the term $\|A+BK\|^i$ instead of $\|A\|^i$. In this case, we can rely on the strong stability assumption, utilized in all OCO-based control works, to bound $\|A+BK\|^i$. The strong stability assumption quantifies the classical stability assumption in control and it states that there exists a strongly stable controller $K$ that is available as an input to our controllers. A strongly stable controller is defined next
\end{itemize}
A linear controller $K$ is $(\kappa, \gamma)$-strongly stable if there exist matrices $L, H$ satisfying $A - BK = HLH^{-1}$, such that the following two conditions are satisfied:
\begin{itemize}
  \item The spectral norm of $L$ satisfies $\|L\|_{\text{op}} \leq 1 -\gamma$.
  \item The controller and transforming matrices are bounded, i.e., $\|K\|_{\text{op}} \leq \kappa$ and $\|H\|_{\text{op}} \|H^{-1}\|_{\text{op}} \leq \kappa$. 
\end{itemize}
Essentially, the existence of $K$ that satisfies the strong stability assumption ensures that we can use the bound 
\begin{align}
    \|A+BK\|_{\text{op}}^t \leq \|H\|_{\text{op}} \|H^{-1}\|_{\text{op}} \|L\|_{\text{op}}^t \leq \kappa (1-\delta)^t.
\end{align}
%maybe change t to i in the above. check later. it has no role except that the power transfers... 

I.e., the norm $\|A+BK\|_{\text{op}}$ decays exponentially and can still be bounded in terms of geometrically decaying terms $(1-\delta)^t$.

\end{document}